\documentclass[11pt,reqno]{amsart}

\usepackage{amsmath,amsthm,amsfonts,amssymb,latexsym,mathrsfs,color,hyperref}

\newtheorem{theorem}{Theorem}
\newtheorem{corollary}[theorem]{Corollary}
\newtheorem{proposition}[theorem]{Proposition}

\newcommand{\des}{{\rm des\,}}

\newcommand{\msn}{\mathfrak{S}_n}

\newcommand{\lrf}[1]{\lfloor #1\rfloor}

\newcommand{\Eulerian}[2]{\genfrac{<}{>}{0pt}{}{#1}{#2}}
\newcommand{\arxiv}[1]{\href{http://arxiv.org/abs/#1}{\texttt{arXiv:#1}}}

\title{On $\gamma$-vectors and the derivatives of the tangent and secant functions}

\author[S.-M.~Ma]{Shi-Mei~Ma}
\address{School of Mathematics and Statistics,
         Northeastern University at Qinhuangdao,
         Hebei 066004, P. R. China}
\email{shimeima@yahoo.com.cn (S.-M. Ma)}
\begin{document}

\maketitle

\begin{abstract}
In this paper we consider the $\gamma$-vectors of the types $A$ and $B$ Coxeter complexes as well as
the $\gamma$-vectors of the types $A$ and $B$ associahedrons. We show that these $\gamma$-vectors can be obtained by using derivative polynomials
of the tangent and secant functions. A grammatical description for these $\gamma$-vectors is discussed. Moreover, we also present a grammatical description for
the well known Legendre polynomials and Chebyshev polynomials of both kinds.
\end{abstract}
\section{Introduction}
Let $\msn$ denote the symmetric group of all permutations of $[n]$, where $[n]=\{1,2,\ldots,n\}$.
The {\it hyperoctahedral group} $B_n$ is the group of signed permutations of the set $\pm[n]$ such that $\pi(-i)=-\pi(i)$ for all $i$, where $\pm[n]=\{\pm1,\pm2,\ldots,\pm n\}$.
A permutation $\pi=\pi(1)\pi(2)\cdots\pi(n)\in\msn$
is {\it alternating} if $\pi(1)>\pi(2)<\pi(3)>\cdots \pi(n)$.
Similarly, an element $\pi$ of $B_n$ is alternating if
$\pi(1)>\pi(2)<\pi(3)>\cdots\pi(n)$.
Denote by $E_n$ and $E_n^B$ the number of alternating elements in $\msn$ and $B_n$, respectively.
It is well known (see~\cite{Chow06,Sta10}) that
$$\sum_{n=0}^\infty E_n\frac{x^n}{n!}=\tan x+\sec x,\quad
\sum_{n=0}^\infty E_n^B\frac{x^n}{n!}=\tan 2x+\sec 2x.$$

Derivative polynomials are an important part of combinatorial trigonometry (see~\cite{Boyadzhiev07,Franssens07,Hoffman95,Hoffman99,Ma12,Ma122,Ma1301} for instance).
Define
$$y=\tan(x),\quad z=\sec(x).$$ Denote by $D$ the differential operator ${d}/{d x}$.
Clearly, we have
$D(y)=1+y^2$ and $D(z)=yz$.
In 1995, Hoffman~\cite{Hoffman95} considered two sequences of {\it derivative polynomials} defined respectively by
$D^n(y)=P_n(y)$ and $D^n(z)=z Q_n(y)$.
From the chain rule it follows that the polynomials $P_n(u)$ satisfy $P_0(u)=u$ and $P_{n+1}(u)=(1+u^2)P_n'(u)$, and similarly $Q_0(u)=1$ and $Q_{n+1}(u)=(1+u^2)Q_n'(u)+uQ_n(u)$.

As shown in~\cite{Hoffman95}, the exponential generating functions
$$P(u,t)=\sum_{n=0}^\infty P_n(u)\frac{t^n}{n!}\quad{\text and} \quad Q(u,t)=\sum_{n=0}^\infty Q_n(u)\frac{t^n}{n!}$$
are given by the explicit formulas
\begin{equation}\label{exponential-deri}
P(u,t)=\frac{u+\tan (t)}{1-u\tan (t)}\quad{\text and} \quad Q(u,t)=\frac{\sec(t)}{1-u\tan (t)}.
\end{equation}

Recall that a {\it descent} of a permutation $\pi\in\msn$
is a position $i$ such that $\pi(i)>\pi(i+1)$, where $1\leq i\leq n-1$. Denote by $\des(\pi)$ the number of descents of $\pi$. Then the equations
\begin{equation*}
A_n(x)=\sum_{\pi\in\msn}x^{\des(\pi)}=\sum_{k=0}^{n-1}\Eulerian{n}{k}x^{k},
\end{equation*}
define the {\it Eulerian polynomial} $A_n(x)$ and the {\it Eulerian number} $\Eulerian{n}{k}$ (see~\cite[A008292]{Sloane}).
For each $\pi\in B_n$,
we define
\begin{equation*}
\des_B(\pi)=\#\{i\in\{0,1,2,\ldots,n-1\}|\pi(i)>\pi({i+1})\},
\end{equation*}
where $\pi(0)=0$. Let
$${B}_n(x)=\sum_{\pi\in B_n}x^{\des_B(\pi)}=\sum_{k=0}^nB(n,k)x^{k}.$$
The polynomial $B_n(x)$ is called an {\it Eulerian polynomial of type $B$}, while $B(n,k)$ is called an {\it Eulerian number of type $B$} (see~\cite[A060187]{Sloane}).

Assume that
$$(Dy)^{n+1}(y)=(Dy)(Dy)^n(y)=D(y(Dy)^n(y)),\quad (Dy)^{n+1}(z)=(Dy)(Dy)^n(z)=D(y(Dy)^n(z)).$$
Note that $D(y)=z^2$ and $D(z)=yz$. Recently, we obtained the following result.
\begin{theorem}[{\cite{Ma1301}}]\label{Ma1301}
For $n\geq 1$, we have
$$(Dy)^n(y)=2^n\sum_{k=0}^{n-1}\Eulerian{n}{k}y^{2n-2k-1}z^{2k+2},\quad
(Dy)^n(z)=\sum_{k=0}^{n}B(n,k)y^{2n-2k}z^{2k+1}.$$
\end{theorem}

Define $$f=\sec(2x),\quad g=2\tan(2x).$$
In this paper we will mainly consider the following differential system:
\begin{equation}\label{diff}
D(f)=fg, \quad D(g)=4f^2.
\end{equation}

Define $h=\tan(2x)$.
Note that $f^2=1+h^2$ and $g=2h$.
So the following result is immediate.
\begin{proposition}\label{Dnf-Qnh}
For $n\geq 0$, we have
$D^n(f)=2^nfQ_n(h),\quad D^n(g)=2^{n+1}P_n(h).$
\end{proposition}

In the next section, we collect some notation and definitions that will be needed in the rest of the paper.
\section{Notation, definitions and preliminaries}\label{notation}
The {\it $h$-polynomial} of a $(d-1)$-dimension simplicial complex $\Delta$ is the generating function
$h(\Delta;x)=\sum_{i=0}^dh_i(\Delta)x^i$ defined by the following identity:
$$\sum_{i=0}^dh_i(\Delta)x^i(1+x)^{d-i}=\sum_{i=0}^df_{i-1}(\Delta)x^i,$$
where $f_{i}(\Delta)$ is the number of faces of $\Delta$ of dimension $i$.
There is a large literature devoted to the
$h$-polynomials of the form
$$h(\Delta;x)=\sum_{i=0}^{\lrf{{d}/{2}}}\gamma_ix^i(1+x)^{d-2i},$$
where the coefficients $\gamma_i$ are nonnegative.
Following Gal~\cite{Gal05}, we call $(\gamma_0,\gamma_1,\ldots)$ the $\gamma$-vector of $\Delta$, and the corresponding
generating function $\gamma(\Delta;x)=\sum_{i\geq 0}\gamma_ix^i$ is the
{\it $\gamma$-polynomial}. In particular, the Eulerian polynomials $A_n(x)$ and $B_n(x)$ are respectively known as
the $h$-polynomials of Coxeter complexes of types $A$ and $B$.

Let us now recall two classical result.
\begin{theorem}[\cite{Foata70,Shapiro83}]\label{Foata}
For $n\geq 1$, we have
\begin{equation*}
A_n(x)=\sum_{k=0}^{\lrf{({n-1})/{2}}}a(n,k)x^k(1+x)^{n-1-2k}.
\end{equation*}
\end{theorem}
\begin{theorem}[{\cite{Chow08,NevoPetersen11,Petersen07}}]\label{Chow}
For $n\geq 1$, we have
\begin{equation*}
B_n(x)=\sum_{k=0}^{\lrf{{n}/{2}}}b(n,k)x^k(1+x)^{n-2k}.
\end{equation*}
\end{theorem}

It is well known that the numbers $a(n,k)$ satisfy the recurrence
\begin{equation*}\label{ank-recu}
a(n,k)=(k+1)a(n-1,k)+(2n-4k)a(n-1,k-1),
\end{equation*}
with the initial conditions $a(1,0)=1$ and $a(1,k)=0$ for $k\geq 1$ (see~\cite[A101280]{Sloane}), and the
numbers $b(n,k)$ satisfy the recurrence
\begin{equation}\label{bnk-recu}
b(n,k)=(2k+1)b(n-1,k)+4(n+1-2k)b(n-1,k-1),
\end{equation}
with the initial conditions $b(1,0)=1$ and $b(1,k)=0$ for $k\geq 1$ (see~\cite[Section 4]{Chow08}).

The $h$-polynomials of the types $A$ and $B$ associahedrons are respectively given as
follows (see~\cite{Mansour09,Marberg12,Postnikov06,Simion03} for instance):
\begin{equation}\label{Anx}
h(\Delta_{FZ}(A_{n-1}),x)=\frac{1}{n}\sum_{k=0}^{n-1}\binom{n}{k}\binom{n}{k+1}x^k
=\sum_{k=0}^{\lrf{{(n-1)}/{2}}}C_k\binom{n-1}{2k}x^k(1+x)^{n-1-2k},
\end{equation}
\begin{equation}\label{Bnx}
h(\Delta_{FZ}(B_n),x)=\sum_{k=0}^{n}{\binom{n}{k}}^2x^k
=\sum_{k=0}^{\lrf{{n}/{2}}}\binom{2k}{k}\binom{n}{2k}x^k(1+x)^{n-2k},
\end{equation}
where $C_k=\frac{1}{k+1}\binom{2k}{k}$ is the $k$th {\it Catalan number} and
the coefficients of $x^k$ of $h(\Delta_{FZ}(A_{n-1}),x)$ is the {\it Narayana number} $N(n,k+1)$.

Define $$F(n,k)=C_k\binom{n-1}{2k},\quad H(n,k)=\binom{2k}{k}\binom{n}{2k}.$$

There are many combinatorial interpretations of the number $F(n,k)$, such as $F(n,k)$ is number of {\it Motzkin paths} of length $n-1$ with $k$ up steps (see~\cite[A055151]{Sloane}).
It is easy to verify that the numbers $F(n,k)$ satisfy the recurrence relation
$$(n+1)F(n,k)=(n+2k+1)F(n-1,k)+4(n-2k)F(n-1,k-1),$$
with initial conditions $F(1,0)=1$ and $F(1,k)=0$ for $k\geq 1$, and the numbers
$H(n,k)$ satisfy the recurrence relation
\begin{equation*}
nH(n,k)=(n+2k)H(n-1,k)+4(n-2k+1)H(n-1,k-1),
\end{equation*}
with initial conditions $H(1,0)=1$ and $H(1,k)=0$ for $k\geq 1$ (see~\cite[A089627]{Sloane}).
\section{$\gamma$-vectors}\label{Section-3}
Define the generating
functions
$$a_n(x)=\sum_{k\geq 0}a(n,k)x^k,\quad b_n(x)=\sum_{k\geq 0}b(n,k)x^k.$$
The first few $a_n(x)$ and $b_n(x)$ are respectively given as follows:
$$a_1(x)=1,a_2(x)=1, a_3(x)=1+2x,a_4(x)=1+8x;$$
$$b_1(x)=1,b_2(x)=1+4x, b_3(x)=1+20x,b_4(x)=1+72x+80x^2.$$

Combining~\eqref{exponential-deri} and~\cite[Prop.~3.5, Prop.~4.10]{Chow08}, we immediately get the following result.
\begin{theorem}\label{thm1}
For $n\geq 1$, we have
$$a_n(x)=\frac{1}{x}\left(\frac{\sqrt{4x-1}}{2}\right)^{n+1}P_n\left(\frac{1}{\sqrt{4x-1}}\right),\quad
b_n(x)=(4x-1)^{\frac{n}{2}}Q_n\left(\frac{1}{\sqrt{4x-1}}\right).$$
\end{theorem}

Assume that
$$(fD)^{n+1}(f)=(fD)(fD)^n(f)=fD((fD)^n(f)),$$
$$(fD)^{n+1}(g)=(fD)(fD)^n(g)=fD((fD)^n(g)).$$

We can now present the main result of this paper.
\begin{theorem}\label{mthm}
For $n\geq 1$, we have
\begin{equation*}
\begin{split}
D^n(f)&=\sum_{k=0}^{\lrf{{n}/{2}}}b(n,k)f^{2k+1}g^{n-2k},\\
D^n(g)&=2^{n+1}\sum_{k=0}^{\lrf{{n-1}/{2}}}a(n,k)f^{2k+2}g^{n-1-2k},\\
(fD)^n(f)&=n!\sum_{k=0}^{\lrf{{n}/{2}}}H(n,k)f^{n+1+2k}g^{n-2k},\\
(fD)^n(g)&=2(n+1)!\sum_{k=0}^{\lrf{{(n-1)}/{2}}}F(n,k)f^{n+2+2k}g^{n-1-2k}.
\end{split}
\end{equation*}
\end{theorem}
\begin{proof}
We only prove the assertion for $D^n(f)$ and the others can be proved in a similar way.
It follows from~\eqref{diff} that $D(f)=fg$ and $D^2(f)=fg^2+4f^3$. For $n\geq 0$, we define $\widetilde{b}(n,k)$ by
\begin{equation}\label{Dnf}
D^n(f)=\sum_{k=0}^{\lrf{{n}/{2}}}\widetilde{b}(n,k)f^{2k+1}g^{n-2k},
\end{equation}
Then $\widetilde{b}(1,0)=1$ and $\widetilde{b}(1,k)=0$ for $k\geq 1$. It follows from~\eqref{Dnf} that
$$D(D^n(f))=\sum_{k=0}^{\lrf{{n}/{2}}}(2k+1)\widetilde{b}(n,k)f^{2k+1}g^{n-2k+1}+
4\sum_{k=0}^{\lrf{{n}/{2}}}(n-2k)\widetilde{b}(n,k)f^{2k+3}g^{n-2k-1}.$$
We therefore conclude that
$\widetilde{b}(n+1,k)=(2k+1)\widetilde{b}(n,k)+4(n+2-2k)\widetilde{b}(n,k-1)$
and complete the proof by comparing it with~\eqref{bnk-recu}.
\end{proof}

Define
\begin{equation*}
\begin{split}
N_n(x)&=\frac{1}{n}\sum_{k=0}^{n-1}\binom{n}{k}\binom{n}{k+1}(x+1)^k(x-1)^{n-1-k},\\
L_n(x)&=\sum_{k=0}^n{\binom{n}{k}}^2(x+1)^k(x-1)^{n-k}.
\end{split}
\end{equation*}

Taking $f^2=1+h^2$ and $g=2h$ in Theorem~\ref{mthm} leads to the following result and we omit the proof of it, since it is a straightforward application of~\eqref{Anx} and~\eqref{Bnx}.
\begin{corollary}\label{cor-1}
For $n\geq 1$, we have
\begin{equation*}
\begin{split}
(fD)^n(f)&=n!f^{n+1}(-\imath)^nL_n(\imath h),\\
(fD)^n(g)&=2(n+1)!f^{n+2}(-\imath)^{n-1}N_n(\imath h),
\end{split}
\end{equation*}
where $\imath=\sqrt{-1}$.
\end{corollary}

It should be noted that the polynomial $\frac{1}{2^n}L_n(x)$ is the famous {\it Legendre polynomial}~\cite[A100258]{Sloane}.
Therefore, from Corollary~\ref{cor-1}, we see that the Legendre polynomial can be generated by $(fD)^n(f)$.
\section{Context-free grammars}\label{Section-4}
Many combinatorial objects permit grammatical interpretations~(see~\cite{Chen93,Chen121,Ma1302} for instance).
The grammatical method was systematically introduced by Chen~\cite{Chen93}
in the study of exponential structures in combinatorics.
Let $A$ be an alphabet whose letters are regarded as independent commutative indeterminates.
A {\it context-free grammar} $G$ over $A$ is defined as a set
of substitution rules that replace a letter in $A$ by a formal function over $A$.
The formal derivative $D$ is a linear operator defined with respect to a context-free grammar $G$.
For example, if $G=\{u\rightarrow uv, v\rightarrow v\}$, then $$D(u)=uv,D(v)=v,D^2(u)=u(v+v^2),D^3(u)=u(v+3v^2+v^3).$$

It follows from Theorem~\ref{mthm} that the $\gamma$-vectors of Coxeter complexes (of types $A$ and $B$) and
associahedrons (of types $A$ and $B$) can be respectively generated
by the grammars $$G_1=\{u\rightarrow uv,v\rightarrow 4u^2\}$$
and
\begin{equation}\label{G2-def}
G_2=\{u\rightarrow u^2v,v\rightarrow 4u^3\}.
\end{equation}

There are many sequences can be generated by the grammar~(\ref{G2-def}).
A special interesting result is the following.
\begin{proposition}\label{pro-Dnuv}
Let $G$ be the same as in~\eqref{G2-def}. Then
\begin{equation*}\label{Dnuv}
D^{n}(uv)=n!\sum_{k=0}^{\lrf{{(n+1)}/{2}}} 4^k\binom{n+1}{2k}u^{n+1+2k}v^{n+1-2k}.
\end{equation*}
\end{proposition}

Let $T_n(x)$ and $U_n(x)$ be the {\it Chebyshev polynomials of the first and second kind} of order $n$, respectively.
We can now conclude the following result, which is based on Proposition~\ref{pro-Dnuv}.
The proof runs along the same lines as that of Theorem~\ref{mthm}.
\begin{theorem}\label{Chebyshev}
If $G=\{u\rightarrow u^2v,v\rightarrow u^3\}$, then
\begin{equation*}
\begin{split}
D^{n}(uv)&=n!\sum_{k=0}^{\lrf{(n+1)/2}}\binom{n+1}{2k}u^{n+1+2k}v^{n+1-2k},\\
D^{n}(u^2)&=n!\sum_{k=0}^{\lrf{n/2}}\binom{n+1}{2k+1}u^{n+2+2k}v^{n-2k}.
\end{split}
\end{equation*}
In particular,
$$D^n(uv)\mid_{u^2=x^2-1,v=x}=n!(x^2-1)^{\frac{n+1}{2}}T_{n+1}(x),$$
$$D^n(u^2)\mid_{u^2=x^2-1,v=x}=n!(x^2-1)^{\frac{n+2}{2}}U_n(x).$$
\end{theorem}

Taking $u={\sec^2(x)}$ and $v=2\tan(x)$, it is clear that $D(u)=uv$ and $D(v)=2u$.
One can easily to verify another grammatical description of the $\gamma$-vectors of the type $A$ Coxeter complex.
\begin{theorem}
If $G=\{u\rightarrow uv, v\rightarrow 2u\}$,
then
\begin{equation*}
D^{n}(u)=\sum_{k=0}^{\lrf{{n}/{2}}}a(n+1,k)u^{k+1}v^{n-2k}.
\end{equation*}
\end{theorem}

Define
\begin{equation}\label{Motzkin}
T(n,k)=\binom{n}{k}\binom{n-k}{\lrf{\frac{n-k}{2}}}.
\end{equation}
It is well known that $T(n,k)$ is the number of paths of length $n$ with steps $U=(1,1),D=(1,-1)$ and $H=(1,0)$, starting at $(0,0)$, staying weakly above the $x$-axis (i.e. left factors of {\it Motzkin paths}) and having $k$ $H$ steps (see~\cite[A107230]{Sloane}).
It follows from~\eqref{Motzkin} that
\begin{equation}\label{Motzkin-recu}
(n+1)T(n,k)=(2n+1-k)T(n-1,k-1)+2T(n-1,k)+4(k+1)T(n-1,k+1).
\end{equation}

We end our paper by giving the following result.
\begin{theorem}\label{thm-motzkin}
If $G=\{t\rightarrow tu^2,u\rightarrow u^2v,v\rightarrow 4u^3\}$,
then
\begin{equation*}
\begin{split}
D^{n}(t^2u^2)&=(n+1)!t^2\sum_{k=0}^{n}T(n,k)u^{2n+2-k}v^{k},\\
D^{n}(t^2u)&=n!t^2\sum_{k=0}^{n}\binom{n}{k}2^{n-k}u^{2n+1-k}v^{k},
\end{split}
\end{equation*}
\end{theorem}
\begin{proof}
We only prove the assertion for $D^{n}(t^2u^2)$ and the corresponding assertion for $D^{n}(t^2u)$ can be proved in a similar way.
It is easy to verify that $D(t^2u^2)=2t^2(u^4+u^3v)$ and $D^2(t^2u^2)=3!t^2(2u^6+2u^5v+u^4v^2)$.
For $n\geq 0$, we define
\begin{equation*}\label{tnk}
D^n(t^2u^2)=(n+1)!t^2\sum_{k=0}^n\widetilde{T}(n,k)u^{2n+2-k}v^{k}.
\end{equation*}
Note that
\begin{equation*}
\begin{split}
\frac{D^{n+1}(t^2u^2)}{(n+1)!t^2}&=\sum_k(2n+2-k)\widetilde{T}(n,k)u^{2n+3-k}v^{k+1}+\\
&2\sum_k \widetilde{T}(n,k)u^{2n+4-k}v^{k}+4\sum_k k\widetilde{T}(n,k)u^{2n+5-k}v^{k-1}.
\end{split}
\end{equation*}
Thus, we get
$$(n+2)\widetilde{T}(n+1,k)=(2n+3-k)\widetilde{T}(n,k-1)+2\widetilde{T}(n,k)+4(k+1)\widetilde{T}(n,k+1).$$
Comparing with~\eqref{Motzkin-recu}, we see that
the coefficients $\widetilde{T}(n,k)$ satisfy the same recurrence relation and initial conditions as $T(n,k)$, so they agree.
\end{proof}

It should be noted that the numbers $\binom{n}{k}2^{n-k}$ are elements of
the $f$-vector for the $n$-dimensional cubes (see~\cite[A038207]{Sloane})


\end{document}